\documentclass[11pt,a4paper]{amsart}
\usepackage{xcolor}
\usepackage[T1]{fontenc}
\usepackage[mathscr]{euscript}
\usepackage[pagebackref,colorlinks=true,linkcolor=blue,citecolor=blue,urlcolor=red, hypertexnames=true]{hyperref}
 \usepackage{caption}
 \usepackage{enumerate}
\usepackage{amsmath,amssymb,bm,amsfonts}

\usepackage{color}

\theoremstyle{plain} 
\newtheorem{thm}{Theorem}[section]
\newtheorem{cor}[thm]{Corollary}

\newtheorem{lem}[thm]{Lemma}

\theoremstyle{definition} 
\newtheorem{definition}[thm]{Definition}

\newtheorem{remarque}[thm]{Remark}

\newcommand{\N}{\mathbb{N}}
\newcommand{\Z}{\mathbb{Z}}

\def\eps{\varepsilon }

\def\bs{\backslash}
\def\l{\left}
\def\r{\right}

\def\X{\mathcal X}

\def\H{\mathcal H}
\def\K{\mathcal K}

\def\arr{\longrightarrow}

\providecommand{\abs}[1]{\left| #1 \right|}
\providecommand{\norm}[1]{\left\|#1\right\|}
\providecommand{\tend}[1]{\underset{ #1 }{\longrightarrow}}
\newcommand{\remove}[1]{ }

\newcommand{\scal}[2]{\ensuremath{{\langle #1,#2\rangle}}}

\numberwithin{equation}{section}

\DeclareMathOperator{\val}{val}

\makeatletter
\@namedef{subjclassname@2020}{%
  \textup{2020} Mathematics Subject Classification}
\makeatother

\author[C. Badea]{Catalin Badea}
\address[C. Badea and O. Devys]{Univ. Lille, CNRS UMR 8524 - Laboratoire Paul
Painlev\'e, F-59000 Lille, France}
\email{cbadea@univ-lille.fr}
\author[O. Devys]{Oscar Devys}

\begin{document}

\title[Rochberg's abstract coboundary theorem]{Rochberg's abstract coboundary theorem revisited}

 \thanks{This work was supported in part by
the project FRONT of the French
National Research Agency (grant ANR-17-CE40-0021), by the Labex CEMPI (ANR-11-LABX-0007-01) and by the Max Planck Institute of Mathematics (Bonn).}

\begin{abstract}  
Rochberg's coboundary theorem provides conditions under which the equation $(I-T)y = x$ is solvable in $y$. Here $T$ is a unilateral shift on Hilbert space, $I$ is the identity operator and $x$ is a given vector. The conditions are expressed in terms of Wold-type decomposition determined by $T$ and growth of iterates of $T$ at $x$. We revisit Rochberg's theorem and prove the following result. Let $T$ be an isometry acting on a Hilbert space $\H$ and let $x \in \H$. Suppose that
$
\sum_{k=0}^\infty k \| T^{*k} x \| < \infty.
$
Then $x$ is in the range of $(I-T)$ if (and only if)
$\l\|\sum_{k= 0}^n T^k x \r\| = o(\sqrt{n}).$  When $T$ is merely a contraction, $x$ is a coboundary under an additional assumption. Some applications to $L^2$-solutions of the functional equation $f(x)-f(2x) = F(x)$, considered by Fortet and Kac, are given.
\end{abstract}

\maketitle

\hfill\emph{To the memory of J\"org Eschmeier}

\par\bigskip
\section{Introduction}\label{Introduction}
\subsection{Coboundaries} Let $T$ be a bounded linear operator acting on a complex Banach space $\X$. An element $x$ of $\X$ is called a \emph{coboundary for} $T$ if there is $y\in \X$ such that $x = y - Ty$. Coboundaries are related to the behavior of the \emph{ergodic sums}
$$ S_n(T) x := x + Tx + \dots + T^{n-1}x, \quad n \ge 1.$$
A variant of the mean ergodic theorem for power bounded operators on reflexive Banach spaces has been proved by von Neumann for Hilbert spaces and by Lorch in the general case~; see for instance \cite{Krengel}. Recall that $T$ is said to be \emph{power bounded} if 
$\sup_{n\ge 1} \|T^n\| < \infty$. We have 
$$ \X = \l\{ x \in\X : \lim_{n\to\infty} \frac{1}{n} S_n(T) x \text{ exists} \r\} 
= \{ y \in \X : Ty = y \} \oplus \overline{(I-T)\X}.$$
In particular, as a consequence of this ergodic decomposition, we have 
$$ x\in \overline{(I-T)\X} \quad \Leftrightarrow \quad \lim_{n\to\infty} \frac{1}{n} S_n(T) x = 0.$$
One can say more about the rate of convergence of $(1/n)S_n(T) x$ to zero when $x$ is a coboundary. Indeed, when there exists a solution $y$ of the equation $y - Ty = x$, the ergodic sums satisfy $S_n(T) x = y - T^ny$. It follows that $(S_n(T)x)_{n\in\N}$ is bounded. Therefore 
\begin{equation} \label{eq:11}
\l\|\frac{1}{n} S_n(T) x\r\| = O \l(\frac{1}{n}\r).
\end{equation}
This rate of convergence to zero, namely $O(1/n)$, characterizes coboundaries of power bounded operators on reflexive spaces. Indeed, the converse result (whenever $T$ is power bounded and $\X$ is reflexive, an element $x$ satisfying \eqref{eq:11} is a coboundary for $T$) has been proved by Browder \cite{Browder} and rediscovered by Butzer and Westphal \cite{BW}.

We also note (see for instance \cite{Gomilko}, \cite{BadeaMuller}, \cite{BadeaGrivauxMuller} and the references therein) that if $(I-T) \X$ is not closed, then for every sequence $(a_n)_{n\ge 1}$ of positive real numbers converging to zero, there exists $x \in \overline{(I-T)\X} \bs (I-T)\X$ such that 
$$  \l\|\frac{1}{n} S_n(T) x\r\| \ge a_n, \quad \forall n \ge 1.$$
In particular, there is no general rate of convergence in the mean ergodic theorem outside coboundaries. 

\subsection{Rochberg's theorem.} Browder's theorem has been extended to the case that $T$ is a dual operator on a dual Banach space by Lin \cite{Lin} ; see also Lin and Sine \cite{LinSine}. We refer the reader to the introduction of \cite{CohenLin}, and the references cited therein, for the history of Browder's theorem and for other extensions and generalizations. We mention here only two references, namely \cite{Robinson} and \cite{Kozma}, dealing with the Hilbert space situation. Any of these Hilbert or Banach space abstract characterizations is not strong enough to obtain as consequences classical results of Fortet and Kac \cite{Fortet, Kac} who dealt with the case $\X = L^2(0,1)$ and $Sf(x) = f(2x)$. This operator $S$ is the Koopman operator associated with the doubling map on the torus~;  see the last section of this manuscript for more information about coboundaries of $S$. This situation has been remedied by Rochberg \cite{Rochberg}, who showed that a condition of $o(\sqrt{n})$ growth of ergodic sums at $x$ is sufficient to ensure that $x$ is a coboundary for a unilateral shift on Hilbert space. Notice that the Koopman operator $S$ acts as a unilateral shift on the subspace of 
$L^2 (0,1)$ of functions whose zeroth Fourier coefficient vanishes.

\smallskip

We need the following classical definition in order to state Rochberg's abstract coboundary theorem. 

\begin{definition}
Let $T$ be an isometry acting on Hilbert space $\H$. A closed subspace $\K$ of $\H$ is called \emph{wandering} for $T$ whenever 
$$ T^p\K \perp T^q \K \quad \text{for} \quad p,q \in \N, p\ne q.$$

The isometry $T$ is called a \emph{(unilateral) shift} if $\H$ possess a closed subspace $\K$, wandering for $T$ and such that
$$ \bigoplus_{n=0}^\infty T^n \K = \H.$$
\end{definition}

\begin{thm}[\cite{Rochberg}]
Let $S$ be a shift and let $f$ be an element of $\H$. Using the notation of the preceding definition, we denote by $f_j$ the projection of $f$ onto the closed subspace $S^j \K$. Suppose that there exists $\beta > 0$ such that
$$ \| f_j\| = O (2^{-\beta j}).$$
Then there exists $g$ in $\H$ such that $(I-S)g = f$ if and only if  
$$ \lim_{n\to \infty} \frac{1}{n} \l\|\sum_{k= 0}^n S^k f \r\|^2 = 0.$$
\end{thm}

\begin{remarque}
The condition 
$$ \| f_j\| = O (2^{-\beta j})$$
 is of course dependent of the decomposition of $\H$ associated with the unilateral shift $S$. It implies $\| S^{* j} f \| = O (2^{-\beta j})$.
\end{remarque}

\subsection{Statement of the main results.} In the next theorem the unilateral shift $S$ is replaced by an arbitrary isometry $T$ and the growth of the norm of the projection $f_j$ by the convergence of the series $\sum_{j=0}^\infty j\| T^{* j} f \|$. The statement of the result does not depend on the Wold decomposition, at least not in an explicit way. For the convenience of the reader, the Wold decomposition theorem is recalled below. Theorem~\ref{t4.6} implies Rochberg's theorem and it allows to recover Kac's results about the coboundaries of the Koopman operator of the doubling map. 

\begin{thm}\label{t4.6}
Let $T$ be an isometry acting on a Hilbert space $\H$ and let $x \in \H$. Suppose that
\begin{equation} \label{eq:summab}
\sum_{k=0}^\infty k \| T^{*k} x \| < \infty.
\end{equation}
Then there exists $y \in \H$ such that $x = (I-T)y$ if and only if
$$ \lim_{n\to \infty} \frac{1}{n} \l\|\sum_{k= 0}^n T^k x \r\|^2 = 0.$$
\end{thm}

Note however that the condition \eqref{eq:summab} implies that $x$ is necessarily an element of the shift part of the isometry $T$. 

Considering coboundaries of adjoints of isometries, we notice that the identity $I-T = (T^*-I)T$ shows that every coboundary of the isometry $T$ is also a coboundary for its adjoint $T^*$. It follows from \cite[Proposition 4.3]{DerLin} that when the isometry $T$ is not invertible (\emph{i.e.}, not a unitary operator), there are coboundaries for $T^*$ which are not coboundaries for $T$.

The following result, more general than Theorem~\ref{t4.6}, is about coboundaries of contractions (operators of norm no greater than one). 

\begin{thm}\label{t4.7}
Let $T$ be a linear operator acting on a Hilbert space $\H$ with $\|T\| \le 1$. Let $x \in \H$ and denote $S_n(T) x := x + Tx + \dots + T^{n-1}x$. Suppose that \eqref{eq:summab} holds, as well as 
\begin{equation} \label{eq:osqrt}
\|S_n(T)x\| = o(\sqrt{n}), \quad n\to \infty 
\end{equation}
and 
\begin{equation} \label{eq:withD}
\sum_{k=1}^{n} \left(\|S_k(T)x\|^2 - \|TS_k(T)x\|^2 \right) = o(n), \quad n\to \infty .
\end{equation}
Then there exists $y \in \H$ such that $x = (I-T)y$. In addition, $y$ can be chosen such that $\|Ty\| = \|y\|$.
\end{thm}

We obtain the following consequence. 

\begin{cor}\label{t4.8}
Let $T$ be a linear operator acting on a Hilbert space $\H$ with $\|T\| \le 1$. Let $x \in \H$ and denote $S_n(T) x := x + Tx + \dots + T^{n-1}x$. Suppose that \eqref{eq:summab} and \eqref{eq:osqrt} hold, as well as 
\begin{equation} \label{eq:kron}
\sum_{k=1}^{\infty} \frac{\left(\|S_k(T)x\|^2 - \|TS_k(T)x\|^2 \right)}{k} < \infty.
\end{equation}
Then there exists $y \in \H$ such that $x = (I-T)y$ and $\|Ty\| = \|y\|$.
\end{cor}

Some remarks are in order. Theorem~\ref{t4.7} and its consequence Corollary~\ref{t4.8} shows that the coboundary equation can be solved within the \emph{maximal isometric subspace} 
$$M = \{x\in \H : \|T^nx\| = \|x\| \text{ for every }n\ge0\}.$$ 
We refer to \cite{Nagy} and \cite{levan} for the canonical decomposition of a contraction into the maximal isometric subspace and its orthogonal. 

Conditions \eqref{eq:withD} and \eqref{eq:kron} are easily verified when $T$ is an isometry. The conditions \eqref{eq:summab} and \eqref{eq:osqrt} are always satisfied when $\|T\| < 1$ ;  however \eqref{eq:kron} is not, unless $x=0$. In fact, $\|Ty\| = \|y\|$ and $\|T\| < 1$ imply that $y=0$ and thus $x=0$. Of course, as $(I-T)$ is invertible when $\|T\| < 1$ by Carl Neumann's lemma, the coboundary equation $x = (I-T)y$ is always solvable in this case.

\subsection{Outline of the paper.} A proof of Theorem~\ref{t4.6} is given in the next section. The more general Theorem~\ref{t4.7} and its consequence Corollary~\ref{t4.8} are proved in Section 3. Some applications to the functional equation $g(x) - g(2x) = f(x)$ are presented in the next section. The last section collects the acknowledgments, a dedication statement and (imposed) conflict of interest and data availability statements.

\section{Proof of Theorem \ref{t4.6}}
We first recall Wold's decomposition Theorem (see \cite[Chapter 1]{Nagy}).

\begin{thm}[Wold decomposition]
Let $T$ be an isometry on a Hilbert $\H$. Then $\H$ decomposes as an orthogonal sum $\H = \H_0 \oplus \H_1$ such that $\H_0$ and $\H_1$ are reducing for $T$, the restriction of $T$ to $\H_0$ is a unitary operator and the restriction of $T$ to $\H_1$ is a unilateral shift (one of the subspaces can eventually reduce to $\{0\}$). This decomposition is unique~; in particular, we have
$$ \H_0 = \bigcap_{n=0}^\infty T^n \H \quad \text{ and }\quad \H_1 = \bigoplus_{n=0}^\infty T^n \K \quad \text{, where } \quad \K = \H \ominus T\H.$$
\end{thm}

\begin{proof}[Proof of Theorem \ref{t4.6}]
If $x = (I-T) y$, then $\sum_{k= 0}^n T^k x = x - T^{n+1}x$. Therefore $\sum_{k= 0}^n T^k x$ is bounded since the isometry $T$ is clearly power-bounded. In particular, 
$$ \lim_{n\to \infty} \frac{1}{n} \l\|\sum_{k= 0}^n T^k x \r\|^2 = 0.$$

Suppose now that
$$ \lim_{n\to \infty} \frac{1}{n} \l\|\sum_{k= 0}^n T^k x \r\|^2 = 0.$$
We want to show the existence of a solution $y$ of the equation $(I-T)y = x$.

Let $\H = \H_0 \oplus \H_1$ be the Wold's decomposition associated with $T$. We notice that $x \in \H_1$. Indeed, if $x = x_0 + x_1$ according to Wold's decomposition of $\H$, then
$$ \lim_{k\to \infty} \| T^{*k} x_1 \| = 0 \quad 
\text{and}\quad \| T^{*n} x_0 \| = \|x_0\|, \quad \forall n \in \N.$$
Therefore
$$\lim_{n\to \infty} \| T^{*n} x \| = \|x_0\|.$$
On the other hand, it follows from \eqref{eq:summab} that
$$\lim_{n\to \infty} \| T^{*n} x \| = 0.$$
We obtain that $x \in \H_1$. In particular, if $\H_1$ is reduced to $\{0\}$, then $x = 0 = (I-T) 0$. Therefore, without loss of any generality, we can assume that $T$ is a shift. 

For each $n \in \N$, we denote by $P_n$ the projection onto the subspace $T^n \K$. For $u \in \H$, we set $u_n:=P_n(u)$, $u^n: = \sum_{j=0}^n u_j$ and $R_n := u-u^n$.

Suppose that $y$ is solution of the equation $(I-T)y = x$. We first obtain, by projecting to $T^k \K$ for each $k \in \N$, the following system of equations :
$$\begin{cases}
   x_0 = y_0 \\
x_1 = y_1- T y_0 \\
\vdots \\
x_k = y_k - T y_{k-1} \\
\vdots
  \end{cases}$$
We then obtain
$$\begin{cases}
   y_0 = x_0 \\
y_1 = x_1 + T y_0 = x_1 + T x_0\\
\vdots \\
y_k = x_k + T y_{k-1} = x_k + T x_{k-1} + \dots + T^{k-1}x_1 + T^k x_0 \\
\vdots
  \end{cases}$$
Consider now, for each $r \in \N$, the element
$$ y_r = \sum_{k=0}^r T^k x_{r-k} \in T^r \K.$$
We will prove that $\sum_{r=0}^\infty \|y_r\|^2$ is convergent, thus showing that $y = \sum_{r=0}^\infty y_r$ is well defined in $\H$. In that case, for every $r\in \N$, we have
\begin{align*}
P_r \big( (I-T) y \big) & = y_r - T y_{r-1} \\
& = \sum_{j=0}^r T^j x_{r-j} - \sum_{j=0}^{r-1} T^{j+1} x_{r-1-j} \\
& = x_r.
\end{align*}
This shows that $(I-T)y = x$.

To prove that $\sum_{r=0}^\infty \|y_r\|^2$ is finite, we need two more results.

\begin{lem}\label{l4.4}
Let $u\in \H$ be such that $\sum_{j\ge 0} \|T^{*j}u \| < +\infty$. 
Then
$$\lim_{n\to \infty} \frac{1}{n} \norm{\sum_{k=0}^n T^k u }^2 = 
\|u\|^2 + 2 Re \sum_{k=1}^\infty \langle u ; T^k u \rangle.$$
\end{lem}

\begin{proof}
We first notice that the sum $\sum_{k=1}^\infty \langle u ; T^k u \rangle$ is absolutely convergent since $( \|T^{*j}u \|)_{j\ge 0}$ is summable. 
For each $n \in \N^*$, we have
\begin{align*} 
 \frac{1}{n}  \norm{ \sum_{k=0}^n T^k u }^2  & = \frac{1}{n} \left(  \sum_{i=0}^n \|T^i u\|^2
+ 2 Re \left( \sum_{0 \le i <j \le n } \langle T^i u ; T^j u \rangle \right) \right) \\
& = \frac{1}{n} \left(  \sum_{i=0}^n \|u\|^2
+ 2 Re \left( \sum_{0 \le i <j \le n } \langle u ; T^{j-i} u \rangle \right) \right) \\
& = \frac{n+1}{n} \|u\|^2 + \frac{2}{n} Re \left( \sum_{r=1}^n (n-r+1)\langle u ; T^r u \rangle \right)  \\
& = \frac{n+1}{n} \|u\|^2 + 2 Re \left( \sum_{r=1}^n \langle u ; T^r u \rangle  - \frac{1}{n}
\sum_{r=1}^n (r-1)\langle u ; T^r u \rangle  \right).
\end{align*} 
On the other hand, we have
$$ \l| \frac{1}{n} \sum_{r=1}^n (r-1)\langle u ; T^r u \rangle\r| 
\le \frac{1}{n} \|u\| \sum_{r=1}^n (r-1) \|T^{*r} u \|.$$
Using again the summability of the sequence $(\|T^{*j}u \|)_{j\ge 0}$ and the Kronecker's lemma (see for instance \cite[Lemma IV.3.2]{Shi}), we get
$$ \frac{1}{n} \sum_{r=1}^n (r-1)\langle u ; T^r u \rangle \tend{n\to \infty} 0.$$
As the series $\sum_{k\ge1} \langle u ; T^k u \rangle$ is convergent, we obtain, as $n$ tends to infinity, 
$$\lim_{n\to \infty} \frac{1}{n} \norm{\sum_{k=0}^n T^k u }^2 = 
\|u\|^2 + 2 Re \sum_{k=1}^\infty \langle u ; T^k u \rangle.$$
\end{proof}

\begin{lem}\label{l4.5}
Let $u \in \H$. For every $r \in \N$ we have
$$ \norm{ \sum_{j=0}^r T^j u_{r-j} }^2 = \lim_{n\to \infty} \frac{1}{n} \norm{ \sum_{j=0}^n T^j u^r}^2.$$
\end{lem}

\begin{proof}
Let $n \ge r$. For  $k \in \N$ we have
$$ P_k \l( \sum_{j=0}^n T^j u^r \r) = 
\begin{cases}
 \sum_{j=0}^k T^j u_{k-j} \quad & \text{if}\quad 0 \le k < r, \\
 \sum_{j=0}^r T^j u_{k-j} \quad & \text{if}\quad r \le k \le n, \\
 \sum_{j=k-n}^r T^j u_{k-j} \quad & \text{if}\quad  n < k \le n+r, \\
 0 & \text{if} \quad k > n+r.
\end{cases}$$
Using the decomposition of $\H$ as $\H = \bigoplus_{n=0}^\infty T^n \K$, we obtain
$$\frac{1}{n} \norm{ \sum_{j=0}^n T^j u^r}^2 
= \frac{1}{n} \sum_{k = 0}^{r-1} \l\| \sum_{j=0}^k T^j u_{k-j} \r\|^2
 + \frac{1}{n} \sum_{k = r}^{n} \norm{ \sum_{j=0}^r T^j u_{k-j} }^2
 + \frac{1}{n} \sum_{k = n+1}^{n+r} \norm{\sum_{j=k-n}^r T^j u_{k-j}}^2.$$
We have
$$ \frac{1}{n} \l( \sum_{k = 0}^{r-1} \l\| \sum_{j=0}^k T^j u_{k-j} \r\|^2 \r) \tend{n\to \infty} 0$$
and
$$\frac{1}{n} \sum_{k = n+1}^{n+r} \norm{\sum_{j=k-n}^r T^j u_{k-j}}^2 = 
\frac{1}{n} \l( \sum_{k = 1}^{r} \norm{\sum_{j=k}^r T^j u_{k-j}}^2 \r) \tend{n\to \infty} 0,$$
as well as
$$\frac{1}{n} \sum_{k = r}^{n} \norm{ \sum_{j=0}^r T^j u_{k-j} }^2 
= \frac{n-r+1}{n} \norm{ \sum_{j=0}^r T^j u_{r-j} }^2 \tend{n\to\infty} \norm{ \sum_{j=0}^r T^j u_{r-j} }^2.$$
We thus obtain
$$ \lim_{n\to \infty} \frac{1}{n} \norm{ \sum_{j=0}^n T^j u^r}^2 = \norm{ \sum_{j=0}^r T^j u_{r-j} }^2 .$$
\end{proof}

We finally show that $\sum_{r \ge 0} \|y_r\|^2 < \infty$. Using Lemma  \ref{l4.5}, we have for each $r \in \N$,
$$ \|y_r\|^2 = \norm{ \sum_{i=0}^r T^i x_{r-i} }^2 
= \lim_{n\to \infty} \frac{1}{n} \norm{ \sum_{i=0}^n T^i x^r}^2.$$
Using the parallelogram identity for the vectors $x^r + R_r = x$, we get
$$ \frac{2}{n} \norm{ \sum_{i=0}^n T^i x^r}^2 =  \frac{1}{n} \norm{ \sum_{i=0}^n T^i x}^2
 + \frac{1}{n} \norm{ \sum_{i=0}^n T^i (x^r - R_r)}^2 \\
 - \frac{2}{n} \norm{ \sum_{i=0}^n T^i R_r}^2. $$
Make now $n$ tends to infinity. Using Lemma \ref{l4.4} for $R_r$ and $x^r - R_r$, and the hypothesis $\frac{1}{n} \norm{\sum_{k=0}^n T^k x }^2 \tend{n\to\infty} 0$, we obtain
\begin{align*}
 2 \|y_r\|^2 & =  \lim_{n\to \infty} \frac{2}{n} \norm{ \sum_{i=0}^n T^i x^r}^2 \\
& =  \lim_{n\to \infty} \frac{1}{n} \norm{ \sum_{i=0}^n T^i (x^r - R_r)}^2 - 2 \lim_{n \to \infty}
\frac{1}{n} \norm{ \sum_{i=0}^n T^i R_r}^2\\
& = \| x^r - R_r \|^2 - 2 \|R_r\|^2 \\
& \quad + 2 Re \sum_{k=1}^\infty \Big( \langle x^r - R_r ; T^k(x^r - R_r) \rangle
 - 2 \langle R_r ; T^k R_r \rangle \Big) \\
&  =  \|x^r\|^2 - \|R_r\|^2  + 2 Re \sum_{k=1}^\infty \Big( \langle x^r ; T^k x^r \rangle - 
\langle x^r ; T^k R_r \rangle \\
& \quad - \langle R_r ; T^k x^r \rangle - \langle R_r ; T^k R_r \rangle \Big) \\
& = \|x^r\|^2 - \|R_r\|^2 + 2 Re \sum_{k=1}^\infty  \langle x^r ; T^k x^r \rangle  - 
2 Re \sum_{k=1}^\infty  \langle R_r ; T^k x \rangle.
\end{align*}
Using now Lemma~\ref{l4.4} applied to $x^r$ and Lemma~\ref{l4.5}, we get 
\begin{align*}
 \|x^r\|^2 + 2 Re \sum_{k=1}^\infty \langle x^r ; T^k x^r\rangle 
& = \lim_{n\to \infty} \frac{1}{n} \norm{\sum_{i=0}^\infty T^i x^r}^2 \\
& = \|y_r\|^2.
\end{align*}
We can infer that
$$ 2 \|y_r\|^2 = \|y_r\|^2 - \|R_r\|^2 - 2 Re \sum_{k=1}^\infty \langle R_r ; T^k x\rangle,$$
so
$$ \|y_r\|^2 = - \|R_r\|^2 - 2 Re \sum_{k=1}^\infty \langle R_r ; T^k x\rangle .$$
For each fixed $r$ we have $R_r = T^{r+1} T^{*(r+1)}x$. Thus $ \|R_r\| = \| T^{*(r+1)}x\|$. As 
$$\sum_{j=1}^{+\infty} j \|T^{*j}x\|< +\infty ,$$ 
we obtain that $(\|R_r\|^2)_r$ is summable.
It suffices to show that
$$\sum_{r=0}^\infty \abs{\sum_{k=1}^\infty \langle R_r ; T^k x \rangle } < \infty.$$
We have
\begin{align*}
\abs{ \sum_{k=1}^\infty \langle R_r ; T^k x \rangle } & =  \abs{ \sum_{k=1}^r \langle R_r ; T^k x \rangle 
+ \sum_{k = r+1}^\infty \langle R_r ; T^k x \rangle } \\
& = \abs{ \sum_{k=1}^r \langle R_r ; T^k x \rangle 
+ \sum_{k=r+1}^\infty \langle R_k ; T^k x \rangle } \\
& \le \sum_{k= 1}^r \|R_r\| \|T^k x\| + \sum_{k=r+1}^\infty \|R_k\| \| T^k x\| \\
& \le \|x\| \left( r \|R_r\| + \sum_{k=r+1}^\infty \|R_k\| \right) \\
& \le \|x\| \left( r \|T^{*(r+1)}x\| + \sum_{k=r+1}^\infty \|T^{*(k+1)}x\| \right).
\end{align*}
Using again the summability of $(r \|T^{*r}x\|)_r$, we get
$$ \sum_{r=0}^\infty \sum_{k=r+1}^\infty \|T^{*k}x\| = \sum_{k=0}^\infty k \|T^{*k}x\| < \infty.$$
Therefore $\sum_{r=0}^\infty \|y_r\|^2 <\infty.$ 
\end{proof}

\section{The case of contractions}

We now prove Theorem~\ref{t4.7} and its consequence Corollary~\ref{t4.8}. 

\begin{proof}[Proof of Theorem \ref{t4.7}]
Let $D$ denote the defect operator $D = (I-T^*T)^{1/2}$, which is well defined since $T$ is a contraction. As
$$ \|Tx\|^2 + \|Dx\|^2 = \scal{T^*Tx}{x} + \scal{(I-T^*T)x}{x} = \|x\|^2,$$
the operator $R : \ell^2(\H) \mapsto  \ell^2(\H)$ given by 
$$ R(x_0, x_1, x_2, \cdots) = (Tx_0, Dx_0, x_1, x_2, \cdots)$$
and with matrix representation 
\begin{equation}
R = 
\begin{bmatrix}
T &  &  \\
D &  & \\
 & I & \\
 & & I \\
 & & & \ddots
\end{bmatrix},
\end{equation}
is an isometry. We can thus apply Theorem~\ref{t4.6} to $R$. 

The iterates of $R$ are given by 
$$ R^k(x_0, x_1, x_2, \cdots) = (T^kx_0, DT^{k-1}x_0, DT^{k-2}x_0,\cdots ,DTx_0, Dx_0, x_1, x_2, \cdots)$$
while their adjoints are given by
$$ R^{*k}(x_0, x_1, x_2, \cdots) = (T^{*k}x_0 + T^{*(k-1)}Dx_1 + \cdots  + T^*Dx_{k-1}, Dx_k, x_{k+1}, x_{k+2}, \cdots).$$
Denote $\tilde{x} = (x,0,0, \cdots) \in \ell^2(\H)$ and $\tilde{y} = (y,y_1,y_2, \cdots) \in \ell^2(\H)$. The equation 
$$\tilde{x} = (I-R)\tilde{y}$$ reduces to the system of equations $x = (I-T)y$, $y_1 = Dy$, $y_2 = y_1$, $y_3 = y_2$, etc. As $\tilde{y} \in \ell^2(\H)$, we obtain $y_1 = y_2 = \cdots = 0$. Therefore the equation $\tilde{x} = (I-R)\tilde{y}$ in $\ell^2(\H)$ is equivalent to
$$ x = (I-T)y \quad \text{and} \quad Dy = 0 .$$
Every positive (i.e. positive semi-definite) operator has the same kernel as its positive square-root ; thus $(I-T^*T)y = 0$. 
Therefore $\|Ty\| = \|y\|$. 

An easy computation shows that the summability condition $\sum_{k=0}^\infty k \| R^{*k} \tilde{x} \| < \infty$ 
is equivalent to $\sum_{k=0}^\infty k \| T^{*k}x \| < \infty$.

Notice now that 
$$ R^k \tilde{x} = R^k(x,0,0, \cdots) = (T^kx, DT^{k-1}x, \cdots, Dx, 0,0, \cdots).$$
Therefore 
$$ \sum_{k=0}^n R^k \tilde{x} = (\sum_{k=0}^n T^kx, D(\sum_{k=0}^{n-1} T^kx), D(\sum_{k=0}^{n-2} T^kx), \cdots, Dx, 0,0, \cdots).$$
Hence, using the notation $S_n(T) x = x + Tx + \dots + T^{n-1}x$, the $o(\sqrt{n})$ condition 
$$\|\sum_{k=0}^{n}R^k\tilde{x}\| = o(\sqrt{n})$$ is equivalent to 
$$ \|\sum_{k=0}^{n}T^k x\| = o(\sqrt{n}) \quad \text{and} \quad \sum_{k=0}^{n} \|D(S_k(T)x)\|^2 = o(n).$$
The proof is now complete using the identity $\|Du\|^2 = \|u\|^2 - \|Tu\|^2$. 
\end{proof}

Corollary~\ref{t4.8} follows from Theorem~\ref{t4.7} and Kronecker's lemma, already used in the proof of Theorem~\ref{t4.6}.

\section{Coboundaries of the doubling map}

Let $\val_2(n)$ be the $2$-valuation of $n$, that is 
$$\val_2(n) = k \quad \text{if} \quad n = m 2^k\quad \text{with}\quad m \notin 2 \Z.$$
For $n\in \Z$, we denote by $ \hat{f}(n) =  \int_0^1 f(t)e^{-int} \, dt$ the $n$-th Fourier coefficient of $f \in L^2 (0,1)$.

\begin{cor}\label{valuation}
Suppose $f$ is a periodic function of period $1$ such that $f \in L^2 (0,1)$,
\begin{equation}\label{eq:41}
\int_0^1 f(t) \, dt = 0
\end{equation}
and there exists $\eps >0$ such that 
\begin{equation}\label{eq:42}
\sum_{n=-\infty}^\infty \val_2 (n)^{4+\eps} \left|\hat{f}(n)\right|^2 < \infty .  
\end{equation}
Then there is a function $g$ in $L^2(0,1)$ of period one such that
$$ f(t) = g(t) - g(2t) \quad a.e. $$
if and only if
$$ \lim_{n\to \infty} \frac{1}{n} \int_0^1 \l| \sum_{i=0}^n f(2^i t)\r|^2 dt = 0$$
\end{cor}

\begin{proof}
We use Theorem \ref{t4.6} applied to the isometry $T : L^2(0,1) \arr L^2(0,1)$ defined by
$$ Tf (t) = f(2t), \quad t \in (0,1)\  \text{mod} \  1.$$
We first remark that condition \eqref{eq:41} is justified by the fact that $T$ acts as a shift operator on the subspace of 
$L^2 (0,1)$ of functions whose zeroth Fourier coefficient vanishes. The condition
$$ \lim_{n\to \infty} \frac{1}{n} \int_0^1 \l| \sum_{i=0}^n f(2^i t)\r|^2 dt = 0$$
is exactly the condition
$$ \lim_{n\to \infty} \frac{1}{n} \l\|\sum_{k= 0}^n T^k f \r\|^2 = 0,$$
which appears in Theorem \ref{t4.6}. 
We want to show that 
$$ \sum_{k=0}^\infty k \| T^{*k} f \| < \infty.$$
Recall that $T$ acts as a shift operator on 
the subspace of $L^2 (0,1)$ of functions whose zeroth Fourier coefficient vanishes. Let $(a_n) = (\hat{f}(n))$ be the sequence of Fourier coefficients of $f$.  We have $a_0 = 0$.
The iterates of the adjoint of $T$ at $f$ can be computed as 
$$ T^{*k}f (t) = \sum_{j=-\infty}^\infty a_{j 2^k} e^{2i\pi j t}, \quad k \in \N.$$

For $\eps >0$, using the change $n = j2^k$ in the order of summation, we get
\begin{align*}
\sum_{k=0}^\infty k \| T^{*k} f \| 
& = \sum_{k=0}^\infty k \l( \sum_{j=-\infty}^\infty | a_{j 2^k} |^2 \r)^{1/2} \\
& = \sum_{k=0}^\infty k^{-(1+\eps)/2} \l(  k^{3+\eps} \sum_{j=-\infty}^\infty | a_{j 2^k} |^2 \r)^{1/2} \\
& \le \l( \sum_{k=0}^\infty k^{-(1+\eps)} \r)^{1/2} \l( \sum_{k=0}^\infty k^{3+\eps} \sum_{j=-\infty}^\infty |a_{j 2^k}|^2 \r)^{1/2} \\
& = \l( \sum_{k=0}^\infty k^{-(1+\eps)} \r)^{1/2} \l( \sum_{n =-\infty}^\infty |a_n|^2 \sum_{k=0}^{\val_2(n)} k^{3+\eps}\r)^{1/2} \\
& \le 2 \l( \sum_{k=0}^\infty k^{-(1+\eps)} \r)^{1/2} \l( \sum_{n =-\infty}^\infty \val_2(n)^{4+\eps} |a_n|^2 \r)^{1/2} .
\end{align*}
Thus, under our hypothesis about the Fourier coefficients, we have
$$ \sum_{k=0}^\infty k \| T^{*k} f \| < \infty.$$
\end{proof}

\begin{cor}\cite{Rochberg}
Let $f$ be a periodic function of period $1$ such that $f \in L^2 (0,1)$,
$$\int_0^1 f(t) \, dt = 0$$
and there exists $\alpha > 0$ such that
\begin{equation}\label{e4.12}
\sum_{k =-\infty}^\infty |\hat{f}((2k+1))2^i|^2 = O (2^{-\alpha i}).
\end{equation}

Then there is a function $g$ in $L^2(0,1)$ of period one such that
$$ f(t) = g(t) - g(2t) \quad a.e. $$
if and only if
$$ \lim_{n\to \infty} \frac{1}{n} \int_0^1 \l| \sum_{i=0}^n f(2^i t)\r|^2 dt = 0$$
\end{cor}

\begin{proof}
The result follows from Corollary \ref{valuation} with $\eps = 1$, say. Indeed, using the condition \eqref{e4.12}, one can estimate
\begin{align*} 
\sum_{n=-\infty}^\infty \val_2 (n)^{5} \left|\hat{f}(n)\right|^2 & = \sum_{i=1}^{\infty}\sum_{k=-\infty}^\infty i^5|\hat{f}((2k+1))2^i|^2\\
 & \lesssim  \sum_{i=1}^{\infty} \frac{i^5}{2^{\alpha i}} < \infty .  
\end{align*}
\end{proof}
 
\begin{remarque}
Condition \eqref{e4.12} is condition (a) from Theorem~4 in \cite{Rochberg}. It has been proved in \cite{Rochberg} that each of other three conditions of Hölder type, called there (b), (c) and (d), implies the condition \eqref{e4.12}. 
Mark Kac has already considered in \cite{Kac} the case when $f$ is in the Hölder class $C^{0,\alpha}$ for some $\alpha > 1/2$. We refer to \cite{Fortet, Cie, Fuku} for other contributions concerning the functional equation $f(t) = g(t) - g(2t)$.
\end{remarque}

\begin{remarque}
All the remarks at the end of the paper \cite{Rochberg} apply also in our situation. In particular, the generalization to the functional equation $f(t) = g(t) - g(nt)$ (for a fixed integer $n$) is immediate. 
\end{remarque}


\section{Declarations}

\subsection{Acknowledgments.} Some of the results presented here are part of the 2012 PhD thesis of the second-named author \cite{Devys} written under the supervision of the first-named author. We wish to thank several persons who encouraged us to present these results to a larger audience and/or to revisit\footnote{As Amor Towles said: ``For as it turns out, one can {\bf revisit} the past quite pleasantly, as long as one does so expecting nearly every aspect of it to have changed''.} them. Special thanks are due to Michael Lin for several interesting comments and remarks. We would like to thank the anonymous referee for a careful reading of the manuscript and very useful suggestions. The first-named author would like to thank the Max Planck Institute for Mathematics in Bonn for providing excellent working conditions and support. 

\subsection{Dedication.} We dedicate the article to the memory of J\"org Eschmeier, a nice person and a master of both abstract and concrete Operator Theory. 


\subsection{Conflict of interest statement.} On behalf of all authors, the corresponding author states that there is no conflict of interest.

\subsection{Data availability statement.} No datasets were generated or analysed during the current study.

\end{document}